\documentclass[leqno]{amsart}

\usepackage{amsfonts}
\usepackage{amssymb, latexsym, amsmath, pb-diagram}
\usepackage{pstricks-add}
\usepackage{lamsarrow, pb-lams}
\usepackage{multicol}
\usepackage{bm}
\usepackage{longtable}
\usepackage[mathscr]{eucal}
\usepackage{epic,eepic}
\usepackage{xy}
\xyoption{all}

\newtheorem{theorem}{Theorem}[section]
\newtheorem{lemma}[theorem]{Lemma}
\newtheorem{proposition}[theorem]{Proposition}

\numberwithin{equation}{section}

\def\qb{\hfill $\Box$}
\def\nr{\refstepcounter{thm}\thethm}
\def\cf{{\it cf.}\ }

\begin{document}

\title[$\beta_{p^2/p^2-1}$]{The secondary periodic element $\beta_{p^2/p^2-1}$ and its applications}
\author[J.~Hong \& X.~Wang]{Jianguo Hong and Xiangjun Wang}
\date{}
\thanks{The authors were supported by NSFC grant No. 11071125, 11261062 and SRFDP No.20120031110025}
\subjclass[2000]{Primary 55Q10, 55T15, 55Txx.}
\address{Jianguo Hong, School of Mathematical Science and LPMC, Nankai University, Tianjin 300071,
P. R. China}
\email{jghong66@163.com}
\address{Xiangjun Wang, School of Mathematical Science and LPMC, Nankai University, Tianjin 300071,
P. R. China}
\email{xjwang@nankai.edu.cn}
\keywords{stable homotopy groups of spheres, Adams-Novikov spectral sequence, infinite descent method}
\maketitle

\begin{abstract}In this paper, we prove that $\beta_{p^2/p^2-1}$ survives to $E_\infty$ in the
Adams-Novikov spectral sequence ($ANSS$) for all $p\geqslant 5$. As an easy consequence, we prove
that $\beta_{sp^2/j}$ is a permanent cycle for $s\geqslant 1$, $j\leqslant p^2-1$. From the Thom
map $\Phi: Ext^{s,t}_{BP_*BP}(BP_*, BP_*)\longrightarrow Ext^{s,t}_A(\mathbb{Z}/p, \mathbb{Z}/p)$,
we also see that $h_0h_3$ survives to $E_\infty$ in the classical Adams spectral sequence.
\end{abstract}

\section{Introduction}

  Let $p\geqslant 5$ be an odd prime. The Adams-Novikov spectral sequence ($ANSS$) based on the
Brown-Peterson spectrum is one of the most powerful tools to compute the $p$-component of the
stable homotopy groups of spheres $\pi_*S^0$ (\cf \cite{adams, Liulevicius, miller, ravenelbook1}).
The $ E_2$-term of the $ANSS$ is $Ext^{s,t}_{BP_*BP}(BP_*, BP_*)$.

  From \cite{Novikov, miller} we know that $Ext^1_{BP_*BP}(BP_*, BP_*)=H^1BP_*$ is generated by $\alpha_{sp^n/n+1}$
for $n\geqslant 0$, $p\nmid s \geqslant 1 $, where $\alpha_{sp^n/n+1}$ has order $p^{n+1}$.
$Ext^2_{BP_*BP}(BP_* $, $ BP_*)=H^2BP_*$ is the direct sum of cyclic groups generated by $\beta_{sp^n/j,i+1}$
for suitable $(n,s,j,i)$ (\cf \cite{miller, ravenelbook1, ravenelbook}).

   It is known that each element $\alpha_{sp^n/n+1}$ in $H^1BP_*$
is a permanent cycle in the $ANSS$ which represents  an element of im$J$ having the same order.
But we are far from fully determining which element of $\beta_{sp^n/j, i+1}$ in $H^2BP_*$ survives to
$E_\infty$.

    Let $\beta_{sp^n/j}$ denote $\beta_{sp^n/j,1}$. H.~Toda \cite{toda1, toda} proved $\alpha_1\beta_1^p$
is zero in $\pi_*S^0$. This relation supports a non-trivial Adams-Novikov differential
$d_{2p-1}(\beta_{p/p})=a\alpha_1\beta_1^p$, which
is called the Toda differential. Based on Toda differential,
D.~Ravenel \cite{ravenel} proved that
$$d_{2p-1} (\beta_{p^n/p^n})\equiv a\alpha_1\beta_{p^{n-1}/p^{n-1}}^p \quad  \text{mod}\quad \text{ker}\
\beta_1^{p(p^{n-1}-1)/(p-1)}$$
for $n \geqslant 1$. That is to say, $\beta_{p^n/p^n}$ can not survive to $E_\infty$ in the $ANSS$.
From this one can see that only  $\beta_{sp^n/j}\in H^2BP_*$ for
$s\geqslant 2$, $1\leqslant j\leqslant p^n$ or
$s=1$, $1 \leqslant j \leqslant p^n-1$  might survive to $E_\infty$ in the $ANSS$.
The following are some known results in this area:

  Oka \cite{oka75} proved that for $s=1$ , $1 \leqslant j \leqslant p- 1$ or $s\geqslant 2$, $1\leqslant j\leqslant p$,
$\beta_{sp/j}$ is a permanent cycle in the $ANSS$.

  Oka \cite{oka77} proved that for $ s=1 $, $ 1 \leqslant j\leqslant 2p-2 $ or $s\geqslant 2$, $1\leqslant j\leqslant 2p$,
  $\beta_{sp^2/j}$ is a permanent cycle in the $ANSS$.

  Later Oka \cite{oka82, oka83} generalized the result to $n\geqslant 2$, i.e. for $n\geqslant 2$;
$s=1$, $1\leqslant j \leqslant 2^{n-1}(p-1)$ or $s\geqslant 2$, $1\leqslant j\leqslant 2^{n-1}p$,
$\beta_{sp^n/j}$ survives to $E_{\infty}$ in the $ANSS$.

  Shimomura  \cite{shimomuraAGT} proved that for $s\geqslant 1$, $1\leqslant j \leqslant p^2-2$,
$ \beta_{sp^2/j} $ survives to $E_\infty$ in the $ANSS$.

  In this paper, we proved:

{\bf Theorem A} {\it Let $p\geqslant 5$ be an odd prime. Then $\beta_{p^2/p^2-1}$ is a permanent cycle
in the Adams-Novikov spectral sequence.}

{\bf Remark:} {\it At the prime $p=5$, D.~Ravenel in \cite{ravenelbook1} proved that $\beta_{5/5}^5=\beta_1x_{952}$
survives to an order $5$ element in $\pi_{990}S^0$ in the $ANSS$. Thus $\beta_{25/24}$ survives to the 4-fold Toda beacket
$\langle \beta_{5/5}^5, \alpha_1, 5, \alpha_1\rangle$. In this paper we will work on the case $p\geqslant 7$. }

 Let $M$ be the cofiber of the degree $p$ map $p : S^0\rightarrow S^0$,
\[
\xymatrix{S^0\ar[r]^p&S^0 \ar[r] & M.}
\]
There exists the Smith-Toda map $v_1^j:\Sigma^{|v_1^j|}M \rightarrow M$  and its cofiber is denoted by $M(1; j)$
$$\xymatrix{\Sigma^{|v_1^j|}M\ar[r]^-{v_1^j}&M \ar[r] &M(1,j)}.$$

D.~Ravenel  proved that

{\bf Theorem 7.12 \cite{ranenel2}} {\it Let $p \geqslant 5$ be an odd prime. If for some fixed $n \geqslant1$,
\begin{enumerate}
\item[(i)] the spectrum $M (1, p^n-1)$ is a ring spectrum,
\item[(ii)] $\beta_{p^n/p^n-1}$ is a permanent cycle and
\item[(iii)] the corresponding homotopy element has order $p$,
\end{enumerate}
then $\beta_{sp^n/j}$ is a permanent cycle (and the corresponding homotopy
element has order $p$) for all $s \geqslant 1$ and $1 \leqslant j \leqslant p^n-1$.}

  From \cite{oka79, oka82, oka83}, we know that $M(1, p^n-1)$ is a ring spectrum for $ n\geqslant 1$, $p\geqslant 5$.
Thus from the  theorem above and Theorem A, we have:

{\bf Proposition  B } {\it Let $p\geqslant 5$ be an odd prime. Then for $s\geqslant 1$, $j\leqslant p^2-1$,
$\beta_{sp^2/j}$ is a permanent cycle in the Adams-Novikov spectral sequence.}

It is not yet known that whether $\beta_{sp^2/p^2}$ for $s\geqslant 2$
 is a permanent cycle.

  Between the $ANSS$ and the classical Adams spectral sequence ($ASS$), there is the Thom reduction map
\[
\Phi : Ext^*_{BP_*BP} (BP_*, BP_*)\longrightarrow Ext^*_{A} (\mathbb{Z}/p, \mathbb{Z}/p)
\]
and $\Phi(\beta_{p^2/p^2-1})=h_0h_3$. Thus

  {\bf Corollary C} {\it  Let $p\geqslant 5$ be an odd prime. Then $h_0h_3$ is a permanent cycle in the classical Adams
  spectral sequence.}

   This paper is arranged as follows. In section 2  we recall the construction of the
topological small descent spectral sequence ($TSDSS$) and the
small descent spectral sequence ($SDSS$), where the $SDSS$ is a spectral sequence that converges to
$Ext_{BP_*BP}^{s,t}(BP_*, BP_*)$ started from the $Ext$ groups of a complex with $p$-cells.
Then we describe the $E_1$-terms of the $SDSS$. In section 3 we compute the Adams-Novikov $E_2$-term
$Ext_{BP_*BP}^{s,t}(BP_*, BP_*)$ subject to $t-s=q(p^3+1)-3$ by the $SDSS$. In section 4, a non-trivial Adams-Novikov
differential $d_{2p-1}(h_{20}b_{11}\gamma_s)=\alpha_1\beta_1^ph_{20}\gamma_s$ is proved. From which
we prove our main theorem by showing that $d_r(\beta_{p^2/p^2-1})=0$ in section 5.

\section{The small descent spectral sequence and the ABC Theorem}

  In 1985, D.~Ravenel \cite{ravenel85, Rav02, ravenelbook1, ravenelbook} introduced the {\it method  of infinite descent}
and used it to  compute the first thousand stems of the stable homotopy groups of spheres at the prime 5.
This method is an approach to finding the $E_2$-term of the $ANSS$ by the following spectral sequence referred as the
{\it small descent spectral sequence} ($SDSS$).

  Hereafter we set that $q=2p-2$. Let $T(n)$ be the Ranevel spectrum (\cf \cite{ravenelbook1} Section 5, Chapter 6)
characterized by
\[
BP_*T(n)=BP_*[t_1, t_2, \cdots, t_n].
\]
Then we have the following diagram
\begin{equation*}
\xymatrix{S^0=T(0)\ar[r] & T(1)\ar[r] & T(2)\ar[r] & \cdots\ar[r] & T(n)\ar[r]& \cdots \ar[r] & BP,
}
\end{equation*}
where $S^0$ denote the sphere spectrum localized at an odd prime $p\geqslant 5$.
Let $T(0)_{p-1}$ and $T(0)_{p-2}$ denote the $q(p-1)$ and $q(p-2)$ skeletons of $T(1)$ respectively,
they are denoted by  $X$ and $\overline{X}$ for simple.
Then
\begin{align*}
X= & S^0\cup_{\alpha_1}e^q\cup\cdots \cup_{\alpha_1}e^{(p-2)q}\cup_{\alpha_1}e^{(p-1)q}
& & \mbox{and} &
\overline{X}= & S^0\cup_{\alpha_1}e^q\cup\cdots\cup_{\alpha_1}e^{(p-2)q}.
\end{align*}
The $BP$-homologies of them are
\begin{align*}
BP_*(X)= & BP_*[t_1]/\langle t_1^p\rangle & & \mbox{and} & BP_*(\overline{X})= & BP_*[t_1]/\langle t_1^{p-1}\rangle.
\end{align*}

    From the definition above we get the following cofibre sequences
\begin{equation}\tag{\nr}
\xymatrix{
S^0 \ar[r]^-{i'} & X\ar[r]^-{j'} & \Sigma^{q}\overline{X}\ar[r]^-{k'} & S^1,
}
\end{equation}
\begin{equation}\tag{\nr}
\xymatrix{
\overline{X} \ar[r]^-{i''} & X\ar[r]^-{j''}& S^{(p-1)q}\ar[r]^-{k''}&\Sigma\overline{X},
}
\end{equation}
and the short exact sequences of $BP$-homologies
\begin{equation}\tag{\nr}
\xymatrix{
0\ar[r] & BP_*S^0 \ar[r]^-{i'_*} & BP_*X\ar[r]^-{j'_*} & BP_*\Sigma^{q}\overline{X}\ar[r] & 0,
}
\end{equation}
\begin{equation}\tag{\nr}
\xymatrix{
0\ar[r] & BP_*\overline{X} \ar[r]^-{i''_*} & BP_*X \ar[r]^-{j''_*} & BP_*S^{(p-1)q}\ar[r]  & 0.
}
\end{equation}

     Put (2.3) and (2.4) together, one has the following long exact sequence
\begin{equation}\tag{\nr}
\xymatrix{
0\ar[r]& BP_*S^0 \ar[r] & BP_*(X)\ar[r] & BP_*(\Sigma^{q}X)\ar[r] & BP_*(\Sigma^{pq}X)\ar[r] & \cdots .
}
\end{equation}
Put (2.1) and (2.2) together, one has the following Adams diagram of cofibres
\begin{equation}\tag{\nr}
\xymatrix{
S^0\ar[d] & \Sigma^{q-1}\overline{X}\ar[l]\ar[d] & S^{pq-2}\ar[l]\ar[d] & \Sigma^{(p+1)q-3}\overline{X}\ar[l]\ar[d]
 & \cdots \ar[l]\\
X & \Sigma^{q-1}X & \Sigma^{pq-2}X & \Sigma^{(p+1)q-3}X. &
}
\end{equation}
Thus one has:

  {\bf Proposition 2.1} {\it (Ravenel \cite{ravenelbook1} 7.4.2 Proposition)
Let $X$ be as above. Then
\begin{enumerate}
\item[(a)] There is a spectral sequence converging to  $Ext_{BP_*BP}^{s+u,*}(BP_*, BP_*S^0)$ with $E_1$-term
\begin{align*}
E_1^{s, t, u}= & Ext^{s, t}_{BP_*BP}(BP_*, BP_*X)\otimes E[\alpha_{1}]\otimes P[\beta_1], \hskip .4cm\mbox{where} \\
 & E_1^{s,t,0}=Ext^{s,t}_{BP_*BP}(BP_*, BP_*X), \hskip .8cm \alpha_{1}\in E_1^{0, q, 1},\hskip .6cm \beta_1\in E_1^{0,qp, 2}
\end{align*}
and $d_r:E_r^{s, t, u}\longrightarrow E_r^{s-r+1, t, u+r}$. Where
$E[-]$ denotes the exterior algebra and $P[-]$ denotes the polynomial algebra on the indicated generators.
This spectral sequence is referred as the small descent spectral sequence ($SDSS$).
\item[(b)] There is a spectral sequence converging to $\pi_*S^0$ with $E_1$-term
\begin{align*}
E_1^{s, t}= & \pi_*(X)\otimes E[\alpha_{1}]\otimes P[\beta_1], \hskip .4cm\mbox{where} \\
 & E_1^{0,t}=\pi_t(X), \hskip .8cm \alpha_{1}\in E_1^{1, q},\hskip .6cm \beta_1\in E_1^{2, pq}
\end{align*}
and $d_r: E_r^{s, t}\longrightarrow E_r^{s+r, t+r-1}$.
This spectral sequence is referred as the
topological small descent spectral sequence ($TSDSS$).
\end{enumerate}
}

The above two spectral sequences produce the $0$-line and the $1$-line
$Ext_{BP_*BP}^{0,*}(BP_*, BP_*S^0)$,
$ Ext_{BP_*BP}^{1,*}(BP_*, BP_*S^0) $ or the
corresponding elements in $\pi_*S^0$ by $Ext^{0, *}_{BP_*BP}(BP_*, BP_*X)$ and $Ext^{1, *}_{BP_*BP}(BP_*, BP_*X)$.
$ Ext_{BP_*BP}^{s,*}(BP_*, BP_*S^0)\ (s\geqslant 2) $ or the
corresponding elements in $\pi_*S^0$  is produced by $Ext^{s, *}_{BP_*BP}(BP_*, BP_*X)$ $ (s \geqslant 2)$ as described as
the following ABC Theorem.

  {\bf ABC Theorem} {\it ( \cite{ravenelbook} 7.5.1 ABC Theorem) For $p>2$ and $t-s<q(p^3+p-1)-3$, $s\geqslant 2$
\[
Ext^{s,t}_{BP_*BP}(BP_*, BP_*X)=A\oplus B\oplus C,
\]
where $A$ is the $\mathbb{Z}/p$-vector space spanned by
\begin{align*}
A= & \left\{\beta_{ip},\ \beta_{ip+1}| i\leqslant p-1\right\}
\cup \left\{\beta_{p^2/p^2-j}|0\leqslant j\leqslant p-1\right\},\\
B= & R\otimes\left\{\gamma_k|k\geqslant 2\right\}
\end{align*}
where
\[
R=P[b_{20}^p]\otimes E[h_{20}]\otimes\mathbb{Z}/p\left\{
\left\{b_{11}^k|0\leqslant k\leqslant p-1\right\}\cup\left\{h_{11}b_{20}^k|0\leqslant k\leqslant p-2\right\}\right\},
\]
and
\[
C^{s,t}=\bigoplus_{i\geqslant 0}R^{s+2i, t+i(p^2-1)q}.
\]
}

 From the generators of $R$ and \cite{ravenel85} 4.11, 4.12 Theorem, we read the generators of $ C $ as follows:

Let $i=jp+m$. Then from $R^{s+\underline{2i}, t+\underline{i}(p^2-1)q}\subset C^{s,t}$ we have:

 \begin{enumerate}
\item $b^{(j+1)p}_{20} \in R^{2(p-m)+2(\underline{jp+m}), t+(\underline{jp+m})(p^2-1)q} \subset C^{2(p-m), t}$,
which is represented by
\[
b^{p-m-1}_{20}u_{jp+m}
\]
for $p -1 \geqslant m \geqslant 1$. From which we get
\[
b^{p-m-1}_{20}u_{jp+m}\otimes E[h_{20}]\otimes
\left\{b_{11}^k|0\leqslant k\leqslant p-1\right\}\cup\left\{h_{11}b_{20}^k|0\leqslant k\leqslant p-2\right\},
\]
where
\[
u_{jp+m} \in C^{2,q\left[(j+1)p^2+(j+m+1)p+m\right]}.
\]
\item $b^k_{11}b^{jp}_{20} \in R^{2(k-m)+2(\underline{jp+m}), t+(\underline{jp+m})(p^2-1)q} \subset C^{2(k-m), t}$,
which is represented by
\[
b^{k-m-1}_{11}\beta_{(j+1)p/p-m}
\]
for $p -1 \geqslant k \geqslant m+1 \geqslant 1$. From which we get
\[
b^{k-m-1}_{11}\beta_{(j+1)p/p-m}\otimes E[h_{20}],
\]
where
\[
\beta_{(j+1)p/p-m} \in C^{2,q\left[(j+1)p^2+jp+m\right]}.
\]
\begin{itemize}
\item Especially $h_{20}b^{p-1}_{11}b^{jp}_{20} \in R^{3+2(\underline{jp+p-2}), t+(\underline{jp+p-2})(p^2-1)q}
   \subset C^{3, t}$
is represented by $h_{11}\beta_{(j+1)p/1,2}$, which is an element of order $p^2$.
\end{itemize}
\item $h_{11}b^k_{20}b^{jp}_{20} \in R^{2(k-m)+1+2(\underline{jp+m}), t+(\underline{jp+m})(p^2-1)q}
\subset C^{2(k-m)+1, t}$, which is represented by
\[
b^{k-m-1}_{20}\eta_{jp+m+1}
\]
for $p -2 \geqslant k \geqslant m+1 \geqslant 1$,
where
\[
\eta_{jp+m+1}=h_{11}u_{jp+m} \in C^{3, q\left[(j+1)p^2+(j+m+2)p+m\right]}.
\]
\item $h_{20}h_{11}b^k_{20}b^{jp}_{20} \in R^{2(k-m+1)+2(\underline{jp+m}), t+(\underline{jp+m})(p^2-1)q} \subset
C^{2(k-m+1)t}$,  which is represented by
\[
b^{k-m}_{20}\beta_{jp+m+2}
\]
for $p -2 \geqslant k \geqslant m \geqslant 0$,
where
\[
\beta_{jp+m+2} \in C^{2,q\left[jp^2+(j+m+2)p+m+1\right]}.
\]
\begin{itemize}
  \item Especially $h_{20}h_{11}b^{p-2}_{20}b^{jp}_{20} \in R^{2+2(\underline{jp+p-2}), t+(\underline{jp+p-2})(p^2-1)q}
  \subset C^{2, t}$ is  represented by
$\beta_{(j+1)p/1,2}$, which is an element of order $p^2$.
\end{itemize}
 \end{enumerate}

    From the ABC Theorem above, we know that $Ext^{s,t}_{BP_*BP}(BP_*, BP_*X)$ for $s\geqslant 2$, $t-s<q(p^3+p-1)-3$ is the
$\mathbb{Z}_{(p)}$-module generated by the following generators, here the generators are listed as generators, total degree $t-s$
and $t-s$ $mod$ $pq-2$ , range of index; where $pq-2=2p^2-2p-2$ is the total degree of $\beta_1 \in E_1^{0,qp,2}$ in the $SDSS$.

\

Generators of A

\begin{longtable}{lllll}
 Generators  & & $t-s$  and $t-s$ $mod$ $pq-2$ & Range of index\vspace{3mm}\\
  $\beta_{ip}$  &        & $q[ip^2+ip-1]-2$ \\
                & $\equiv$ & $2(i-1)p+2i$ & if $i\leqslant p-2 $ \\
                 &$\equiv$ & $0$                                    & if $i=p-1$ \vspace{3mm} \\[0.6ex]
  $\beta_{ip+1}$ &         & $q[ip^2+(i+1)p]-2$\\
                 &$\equiv$ & $2ip+2i$                                & if $ \ i\leqslant p-2$\\
                 &$\equiv$ & $\underline{2p }_{ 2p}$       & if $ \ i=p-1 $\vspace{3mm}\\[0.6ex]
 $\beta_{p^2/p^2-j}$ &     & $q[p^3+j]-2$ \\
                  &$\equiv$& $\underline{2(j+1)p-2j}_{ 2p}$ & if $\  j \leqslant p-2 $ \\
                  &$\equiv$& $4$                                     & if $\  j=p-1$\\
\end{longtable}

Generators of B

\begin{longtable}{lll}
 Generators  & $t-s$ and $t-s$ mod $pq-2$ & Range of index\\[0.6ex]
\endfirsthead
 Generators  & $t-s$ and $t-s$ mod $pq-2$ & Range of index\\[0.6ex]
\endhead
\multicolumn{2}{r}{continue to next page\dots} \\
\endfoot
\endlastfoot
  $h_{1 1}b_{2 0}^k\gamma_i$ & $\ q[(i+k)p^2+(i+k)p+i-2]-2k-4$  &  \\
                                &  $\equiv 2(k+2i-2)p$ & if  $2\leqslant i$,\ $ k+2i\leqslant p$\\
                              &$\equiv 2(k+2i-p-1)p+2$ & if $k+2i> p$\vspace{3mm}\\[0.6ex]
  $h_{2 0}h_{1 1}b_{2, 0}^k\gamma_i$  & \multicolumn{2}{l}{ $\ \ q[(i+k)p^2+(i+k+1)p+i-1]-2k-5$  }  \\
                                         &  $\equiv 2(k+2i-1)p-1$ & if $k+2i< p$\\
                                        &$\equiv 2(k+2i-p)p+1$ & if $ k+2i\geqslant p$\vspace{3mm}\\[0.6ex]
 $b_{1 1}^k\gamma_i$    & $\ \ q[(i+k)p^2+(i-1)p+i-2]-2k-3$  &  \\
                        & $\equiv 2(k+2i-2)p-2k-1$ & if  $k+2i\leqslant p+1$\\
                        & $\equiv 1$   &  if $k=0$, $2i=p+1$\\
                        & $\equiv \underline{2(k+2i-p-1)p-2k+1}_{4p-3}$ & if $k+2i\geqslant p+2$\vspace{3mm}\\[0.6ex]
 $h_{2 0}b_{1 1}^k\gamma_i$   & $\ \ q[(i+k)p^2+ip+i-1]-2k-4$  &  \\
                              & $\equiv 2(k+2i-1)p-2(k+1)$ & if $k+2i\leqslant p$\\
                              & $\equiv \underline{2(k+2i-p)p-2k}_{2p}$ & if $k+2i> p$\\
\end{longtable}

Generators of C 

\begin{longtable}{lll}
 Generators  & $t-s$ and $t-s$ mod $pq-2$& Range of index\\[0.6ex]
\endfirsthead
 Generators  & $t-s$ and $t-s$ mod $pq-2$& Range of index\\[0.6ex]
\endhead
\multicolumn{2}{r}{continue to next page\dots} \\
\endfoot
\endlastfoot
$b_{1 1}^{k}b_{2 0}^{p-m-1}u_{jp+m}$ &\multicolumn{2}{l}{$\ \ q[(p-m+j+k+1)p^2+jp+m]-2(p-m+k)$  }  \\[1ex]
                                     & $\equiv 2(j+k+1)p+2(j-k+1) $ & for $ 1\leqslant m\leqslant p-1$ \vspace{3mm}\\[1ex]
\multicolumn{3}{l}{$h_{2 0}b_{1 1}^{k}b_{2, 0}^{p-m-1}u_{jp+m}$ $\ \ q[(p-m+j+k+1)p^2+(j+1)p+m+1]-2(p-m+k)-1$ }   \\[1ex]
                                  & $\equiv 2(j+k+2)p+2(j-k+1)-1$ & if $ j+k\leqslant p-4$ \\
                                  &  &  or $j+k=p-3,2j<p-5$\\
                                  & $\equiv 2(j-k+2)p-1$ & if $ j+k=p-3, 2j \geqslant p-5$ \\[1.6ex]
\multicolumn{3}{l}{$h_{1 1}b_{2 0}^{k+p-m-1}u_{jp+m}$ $\ \ q[(p-m+j+k+1)p^2+(j+k+1)p+m]-2(p-m+k)-1$  }   \\[1ex]
                            & $\equiv 2(j+k+2)p+2(j-p)+3$ & for $2\leqslant m\leqslant p-1$ \\[1ex]
\multicolumn{3}{l}{$h_{2 0}h_{1 1}b_{2, 0}^{k+p-m-1}u_{jp+m}$ $\ \ q[(p-m+j+k+1)p^2+(j+k+2)p+m+1]-2(p-m+k+1)$ }   \\[1ex]
                               & $\equiv 2(j+k+2)p+2j+2$ & if $j+k\leqslant  p-4$  \\
                               & $\equiv 2j+4$ & if $j+k = p-3$\\[1.6ex]
$b_{1 1}^{k-m-1}\beta_{(j+1)p/p-m}$ & \multicolumn{2}{l}{\ \ \ $q[(j+k-m)p^2+jp+m]-(2k-2m)$}    \\[1ex]
                              & $\equiv 2(j+k)p+2(j-k)$ &if \ $j+k\leqslant p-2 $\\
                              &      &or $ j+k =p-1$, $2j < p-1 $\\
                             & $\underline{\equiv 2(j+k-p+1)p+2(j-k+1)}_{2p} $ &if $ j+k\geqslant p $\\
                              & & or $ j+k =p-1$, $2j\geq p-1 $\vspace{3mm}\\[1ex]
\multicolumn{3}{l}{$h_{2 0}b_{1 1}^{k-m-1}\beta_{(j+1)p/p-m}$ $\ \  q[(j+k-m)p^2+(j+1)p+m+1]-(2k-2m+1)$  }  \\[1ex]
                  &$\underline{\equiv 2(j+k+1)p+2(j-k)-1}_{4p-3}$ & if $ j+k\leqslant p-3 $\\
                  & & or $ j+k =p-2$, $2j\leqslant p-3$\\
                  &$\equiv 2(j+k-p+2)p+2(j-k)+1 $&if $j+k > p-2 $ \\
       &  &  or $j+k =p-2$, $2j > p-3$ \\
       & $\equiv 1$ &  if $j=p-2$, $k =p-1$\vspace{3mm}\\[0.6ex]
$h_{1, 1}\beta_{(j+1)p/1, 2}$& \multicolumn{2}{l}{\ \ \ $q[(j+1)p^2+(j+2)p-1]-3$  }  \\[0.6ex]
                                         & $\equiv 2jp+2(j+1)+1$& if $j\leqslant p-3$\\
                                         &$\equiv 1$& if $ j=p-2$\vspace{3mm}\\[0.6ex]
$ b_{2, 0}^{k-m-1}\eta_{jp+m+1}$ &\multicolumn{2}{l}{$\ \ q[(j+k-m)p^2+(j+k+1)p+m]-(2k-2m+1)$ }   \\[0.6ex]
                         & $\equiv 2(j+k)p+2j+1$ & if $j+k \leqslant p-2$\\
                         &$\underline{\equiv 2(j+k-p+2)p+2(j-p)+3}_{4p-3}$ & if  $ j+k > p-2$\vspace{3mm}\\
$b_{2, 0}^{k-m}\beta_{jp+m+2}$ &\multicolumn{2}{l}{$\ \ q[(j+k-m)p^2+(j+k+2)p+m+1]-2(k-m+1)$ }   \\[0.6ex]
                         &  $\underline{\equiv 2(j+k+1)p+2j}_{2p}$ & if $j+k\leqslant p-3$\\
                        &$\equiv 2(j+k-p+3)p+2(j-p)+2 $ & if $j+k > p-3$\\
                        &$\equiv 0$ & if $j=k=p-2$\vspace{3mm}\\[0.6ex]
$\beta_{(j+1)p/1, 2}$& \multicolumn{2}{l}{\ \ \ $q[(j+1)p^2+(j+1)p-1]-2$  }  \\[0.6ex]
                                         & $\equiv 2jp+2(j+1)$& if $j\leqslant p-3$\\
                                         &$\equiv 0$& if $ j=p-2$\\[0.6ex]
\end{longtable}
\textbf{Remark.}
The Adams-Novikov spectral sequence for the spectrum $ X $  collapses from $E_2$-term
$Ext^{s, t}_{BP_*BP}(BP_*, BP_*X)$  in the range  $t-s<q(p^3+p-1)-3$, since there are no elements with filtration $ > 2p$.
Thus we actually get the homotopy groups $\pi_{t-s}(X)$ in this range.

\section{The $ANSS$ $E_2$-term $Ext_{BP_*BP}^{s,t}(BP_*, BP_*)$ at $t-s=q(p^3+1)-3$}

    Consider the Adams-Novikov differential $d_r: E_r^{s,t}\rightarrow E_r^{s+r, t+r-1}$ in the $ANSS$.
From the total degree of $\beta_{p^2/p^2-1}$, we know that $d_r(\beta_{p^2/p^2-1})\in Ext_{BP_*BP}^{s, t}(BP_*, BP_*)$
such that $t-s=q(p^3+1)-3$.
The $SDSS$ $E_1^{s,t,u}$ converges to $Ext_{BP_*BP}^{s+u,t}(BP_*, BP_*)$. Fix $t-s-u=q(p^3+1)-3$, we have:
\begin{lemma} Fix $t-s-u=q(p^3+1)-3$, the $E_1$-term $E_1^{s,t,u}$ of the $SDSS$ is the $\mathbb{Z}/p$-module
generated by the following $8$ generators:
\begin{align*}
\mathfrak{g}_1= & \alpha_1\beta_1^{p^2-1}\beta_2 \in E_1^{2,*,2p^2-1}; &
\mathfrak{g}_2= &\beta_1^{p^2-p}h_{20}\beta_{p/p}\in E_1^{3,*,2p^2-2p};\\
\mathfrak{g}_3= & \alpha_1\beta_1^{\frac{p^2-2p-1}{2}}h_{2,0}\gamma_{\frac{p+1}{2}} \in E_1^{4,*,p^2-2p}; &
\mathfrak{g}_4= & \beta_1^{\frac{p^2-6p+1}{2}}b_{11}^{2}\gamma_{\frac{p+1}{2}} \in E_1^{7, *, p^2-6p+1}; \\
\mathfrak{g}_5= & \alpha_1\beta_1^{mp-\frac{p-1}{2}}b_{11}^{\frac{p-1}{2}-m}\beta_{(\frac{p+1}{2})p/p-m}
                      \in E_1^{p+1-2m, *, *}; &
\mathfrak{g}_6= & \beta_1^{p-1} \eta_{(p-3)p+3} \in E_1^{3,*, 2p-2};\\
\mathfrak{g}_7= & \alpha_1\beta_{(p-1)p+1}\in E_1^{2,q(p^3+1),1}; &
\mathfrak{g}_8= & \alpha_1\beta_{p^2/p^2} \in E_1^{2,q(p^3+1),1}.
\end{align*}
\end{lemma}

\begin{proof}
  Fix $t-s-u=q(p^3+1)-3$. From the ABC Theorem, we know that the generators of the $E_1$-terms in the $SDSS$ are of the form
$W=\beta_1^k w$
or $W=\alpha_1 \beta_1^k w,$ where $ w $ is an element listed in the ABC Theorem.

{\bf 1.} If a generator of  $E_1^{s,t,u}$ is of the form $W=\beta_1^k w$, then the total degree of $\beta_1^pw$
is $q(p^3+1)-3$ and the total degree of $w$ is $q(p^3+1)-3$ modulo the total degree of $\beta_1$ which is $t-u=qp-2$.
Note that
\[
q(p^3+1)-3\equiv 4p-3 \hspace{10mm} \mbox{$mod$ $qp-2$,}
\]
we list all the generators whose total degree might be $4p-3$ $mod$ $qp-2$,
which are marked with underline and subscript $4p-3$ in the table for ABC Theorem.
\begin{align*}
 &  \mbox{$b_{11}^k\gamma_i$} & & \mbox{at $k=2$ and $i=(p+1)/2$;}\\
 &  \mbox{$h_{20}b_{11}^{k-m-1}\beta_{(j+1)p/p-m}$} & & \mbox{at $k=1$ and $j=0$;}\\
 &  \mbox{$b_{20}^{k-m-1}\eta_{jp+m+1}$} & & \mbox{at $k=3$ and $j=p-3$.}
\end{align*}
From which we get the following generators in $E_1^{s, t, u}$:
\begin{align*}
& b_{11}^2\gamma_{\frac{p+1}{2}} & \Longrightarrow & & &
\mathfrak{g}_4=  \beta_1^{\frac{p^2-6p+1}{2}}b_{11}^{2}\gamma_{\frac{p+1}{2}}\in E_1^{7,*,p^2-6p+1}; \\
& h_{20}\beta_{p/p} & \Longrightarrow & & & \mathfrak{g}_2= \beta_1^{p^2-p}h_{20}\beta_{p/p}\in E_1^{3,*,2p^2-2p};\\
& \eta_{(p-3)p+3} & \Longrightarrow & & & \mathfrak{g}_6=  \beta_1^{p-1} \eta_{(p-3)p+3}\in E_1^{3,*,2p-2}.
\end{align*}

{\bf 2.} If a generator of $E_1^{s,t,u}$ is of the form $W=\alpha_1 \beta_1^k w_1 $, then
from the total degree of $\alpha_1$ being $t-u=2p-3$ we see that the total degree of
$w_1$ is $2p$ modulo $qp-2$. Similarly we can find all such  $w_1$'s, which are marked with underline
and subscript $2p$ in the ABC Theorem:
\begin{align*}
& \beta_{(p-1)p+1}; & & \beta_{p^2/p^2}; & & h_{20}\gamma_{\frac{p+1}{2}}; & & b_{11}^{\frac{p-1}{2}-m}\beta_{(\frac{p+1}{2})p/p-m}; &
&  \beta_{2}.
\end{align*}
From which we get the following generators in $E_1^{s,t,u}$:
\begin{align*}
\mathfrak{g}_7= & \alpha_1\beta_{(p-1)p+1} ; &
\mathfrak{g}_8= & \alpha_1\beta_{p^2/p^2};   \\
\mathfrak{g}_3= & \alpha_1\beta_1^{\frac{p^2-2p-1}{2}}h_{2,0}\gamma_{\frac{p+1}{2}}; &
\mathfrak{g}_5= & \alpha_1\beta_1^{mp-\frac{p-1}{2}}b_{11}^{\frac{p-1}{2}-m}\beta_{(\frac{p+1}{2})p/p-m}; \\
\mathfrak{g}_1= & \alpha_1\beta_1^{p^2-1}\beta_2.
\end{align*}
Computing the filtration of the corresponding generators, we get the lemma.
\end{proof}

\begin{theorem}
Fix $t-s=q(p^3+1)-3$, the Adams-Novikov $E_2$-term $Ext^{s,t}_{BP_*BP}(BP_*, BP_*)$ is the $\mathbb{Z}/p$-module generated by the following 6 elements
\begin{align*}
\mathfrak{g}_1= & \alpha_1\beta_1^{p^2-1}\beta_2 \in Ext^{2p^2+1,*}_{BP_*BP}; \\
\mathfrak{g}_3= & \alpha_1\beta_1^{\frac{p^2-2p-1}{2}}h_{2,0}\gamma_{\frac{p+1}{2}} \in Ext_{BP_*BP}^{p^2-2p+4,*}; &
\mathfrak{g}_4= & \beta_1^{\frac{p^2-6p+1}{2}}b_{11}^{2}\gamma_{\frac{p+1}{2}} \in Ext_{BP_*BP}^{p^2-6p+8,*}; \\
& & \mathfrak{g}_6= & \beta_1^{p-1} \eta_{(p-3)p+3} \in Ext_{BP_*BP}^{2p+1,*};\\
\mathfrak{g}_7= & \alpha_1\beta_{(p-1)p+1}\in Ext^{3,q(p^3+1)}; &
\mathfrak{g}_8= & \alpha_1\beta_{p^2/p^2} \in Ext^{3,q(p^3+1)}.
\end{align*}
\end{theorem}

\begin{proof}
Following D.~Ravenel \cite{ravenelbook1} page 287, we compute in the cobar complex of $N_0^2=BP_*/(p^\infty, v_1^\infty)$
\begin{align*}
d\left(\frac{v_2^{jp}}{pv_1^p}(t_2-t_1^{p+1})\right) = & \frac{v_2^{jp}}{pv_1^p}t_1^p\otimes t_1 +\frac{v_2^{jp}}{pv_1^{p-1}}b_{10},\\
-d\left(\frac{v_2^{jp+1}}{pv_1^{p+1}}t_1\right)= & -\frac{v_2^{jp}}{pv_1^p}t_1^p\otimes t_1 -j\frac{v_2^{(j-1)p+1}}{pv_1}t_1^{p^2}\otimes t_1
 +\frac{v_2^{jp}}{pv_1}t_1\otimes t_1, \\
d\left(j\frac{v_2^{(j-1)p}v_3}{pv_1}t_1\right) = & j\frac{v_2^{(j-1)p+1}}{pv_1}t_1^{p^2}\otimes t_1 -j\frac{v_2^{jp}}{pv_1}t_1\otimes t_1, \\
-(j-1)/2d\left(\frac{v_2^{jp}}{pv_1}t_1^2\right)= & (j-1)\frac{v_2^{jp}}{pv_1}t_1\otimes t_1.
\end{align*}
Straightforward calculation shows that the coboundary of
\[
\frac{v_2^{jp}}{pv_1^p}t_2-\frac{v_2^{jp}}{pv_1^p}t_1^{p+1}-\frac{v_2^{jp+1}}{pv_1^{p+1}}t_1+
j\frac{v_2^{(j-1)p}v_3}{pv_1}t_1-(j-1)/2\frac{v_2^{jp}}{pv_1}t_1^2
\]
is $\displaystyle{\frac{v_2^{jp}}{pv_1^{p-1}}b_{10}}$. Then from $\delta\delta\displaystyle{\left(\frac{v_2^{jp}}{pv_1^p}\right)}=\beta_{jp/p}$,
we get a differential in the $SDSS$
\begin{align*}
d_2(h_{20}\beta_{jp/p})=\beta_1\beta_{jp/p-1}.
\end{align*}
Similarly, we have
\begin{align}\tag{3.1}
d_2(h_{20}\beta_{jp/i})= & \beta_1\beta_{jp/i-1} & & \mbox{for $2\leqslant i\leqslant p$.}
\end{align}
Applying formula (3.1), we get the following differentials in the $SDSS$
\begin{align*}
d_2(\mathfrak{g}_2) = d_2(\beta_1^{p^2-p}h_{20}\beta_{p/p}) = & \beta_1^{p^2-p+1}\beta_{p/p-1}, \\
d_2(\alpha_1\beta_1^{mp-\frac{p-1}{2}-1}b_{11}^{\frac{p-1}{2}-m}h_{20}\beta_{(\frac{p+1}{2})p/p-m+1}) = &
\alpha_1\beta_1^{mp-\frac{p-1}{2}}b_{11}^{\frac{p-1}{2}-m}\beta_{(\frac{p+1}{2})p/p-m}= \mathfrak{g}_5
\end{align*}
{\small
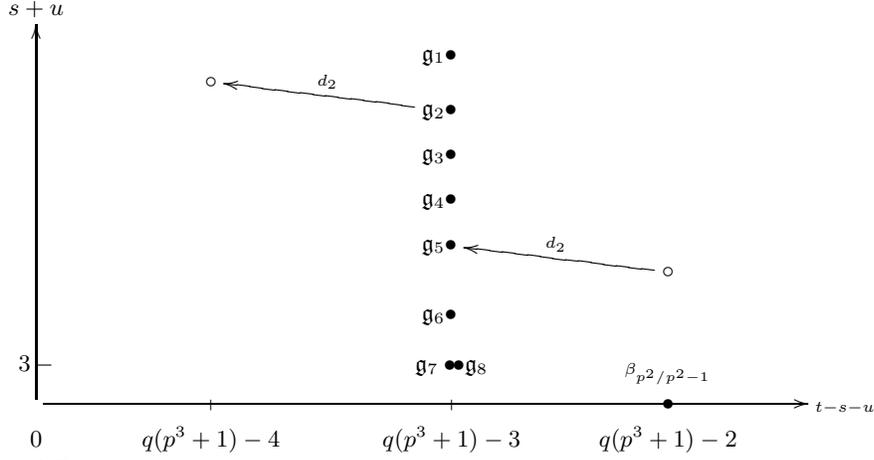
\xymatrix@R=0pt{
&&s+u&&&& \\
&&&&&&\\
&&&&\mathfrak{g}_1\bullet&&\\
&&&\circ&&&\\
&&&&\mathfrak{g}_2\bullet\ar[ul]_{d_{2}}&&\\
&&&&&&\\
&&&&\mathfrak{g}_3\bullet&&\\
&&&&&&\\
&&&&\mathfrak{g}_4\bullet&&\\
&&&&&&\\
&&&&\mathfrak{g}_5\bullet&&\\
&&&&&\circ\ar[ul]_{d_{2}}&\\
&&&&&&\\
&&&&\mathfrak{g}_6\bullet&&\\
&&&&&&\\
&&3\ \! -&&\quad\mathfrak{g}_7\bullet\!\!\bullet \mathfrak{g}_8 & _{\beta_{p^2/p^2-1}}&\\
&&\ar[uuuuuuuuuuuuuuuu]\ar[rrrr]& \shortmid&\quad \shortmid& \bullet & _{t-s-u}\\
&&0&q(p^3+1)-4&\quad q(p^3+1)-3&q(p^3+1)-2&}
}

The theorem follows.
\end{proof}

\section{A differential in the $ANSS$}

  This section is armed at showing that
\begin{align}
d_{2p-1}(h_{20}b_{11}\gamma_{s})=\alpha_1\beta_1^{p}h_{20}\gamma_{s}
\end{align}
in the $ANSS$, which will be used in proving Theorem A in section 5.

  We begin from showing that $\pi_{q(p^2+2p+2)-2}V(2)=0$.
From which we show that the Toda bracket $\langle \alpha_1\beta_1, p, \gamma_s \rangle=0$ and
the Toda bracket $\langle\alpha_1\beta_1^{p-1}, \alpha_1\beta_1, p, \gamma_s \rangle$
is well defined. Then from the relation
\[
\langle\alpha_1\beta_1^{p-1}, \alpha_1\beta_1, p, \gamma_s \rangle=\alpha_1\beta_1^{p-1}h_{20}\gamma_s  =\beta_{p/p-1}\gamma_s
\]
in $\pi_*S^0$ and $d(h_{20}b_{11})=\beta_1\beta_{p/p-1}$, we get the desired differential in the $ANSS$.

    Let $p\geqslant 5$ be an odd prime and $V(2)$ be the Smith-Toda spectrum characterized by
\[
BP_*V(2)=BP_*/I_3
\]
where $I_3$ is the invariant ideal of $BP_*=\mathbb{Z}_{(p)}[v_1, v_2, \cdots, v_i, \cdots]$ generated by
$ p, v_1$ and $v_2 $. To compute the homotopy groups of $V(2)$, one has the $ANSS$
$\{E_r^{s,t}V(2), d_r\}$ that converges to $\pi_*V(2)$. The $E_2$-term of this spectral sequence
is
\[
E_2^{s,t}V(2)=Ext_{BP_*BP}^{s,t}(BP_*, BP_*V(2))
\]

    Let
\[
\Gamma=BP_*/I_3\otimes_{BP_*}BP_*BP\otimes_{BP_*}BP_*/I_3=BP_*/I_3[t_1, t_2, \cdots].
\]
Then $(BP_*/I_3, \Gamma)$ is a Hopf algebroid, and its structure map is deduced from that of
$(BP_*, BP_*BP)$.  By a change of ring theorem, one sees that
\[
Ext_{BP_*BP}^{s,t}(BP_*, BP_*V(2))=Ext_{\Gamma}^{s,t}(BP_*, BP_*/I_3) \Longrightarrow \pi_*V(2)
\]

\begin{lemma}\label{theorem0} The $q(p^2+2p+2)-2$ dimensional stable homology group of $V(2)$ is trivial, i.e.,
$$\pi_{q(p^2+2p+2)-2}V(2)=0.$$
\end{lemma}
\begin{proof}
    Fix $t-s=q(p^2+2p+2)-2$, we know that the Adams-Novikov $E_2$-term
\[
Ext_{BP_*BP}^{s,s+q(p^2+2p+2)-2}(BP_*, BP_*V(2))=Ext_\Gamma^{s,s+q(p^2+2p+2)-2}(BP_*, BP_*/I_3)
\]
converges to $\pi_{q(p^2+2p+2)-2}V(2)$. We will prove that $\pi_{q(p^2+2p+2)-2}V(2)=0$ by showing that
$Ext_{BP_*BP}^{s,s+q(p^2+2p+2)-2}(BP_*, BP_*V(2))=0$.

  In the cobar complex $C^s_\Gamma BP_*/I_3$, the inner degree of $v_i$, $|v_i|=|t_i| \geqslant q(p^3+p^2+p+1)$
for $i \geqslant 4$. It follows that in the range $t-s\leqslant q(p^3+p^2+p+1)-1$,
\[
Ext_{BP_*BP}^{s,t}(BP_*, BP_*/I_3)=Ext_\Gamma^{s,t}(BP_*, BP_*/I_3)=Ext_{\Gamma'}^{s,t}(BP_*, BP_*/I_3).
\]
where $\Gamma' = \mathbb{Z}/p[v_3][t_1, t_2, t_3]$. From $\eta_R(v_3)\equiv v_3$ $mod$ $I_3$, we see that
\[
Ext_{\mathbb{Z}/p[v_3][t_1, t_2, t_3]}^{s,*}(BP_*, BP_*/I_3)\cong
Ext_{\mathbb{Z}/p[t_1, t_2, t_3]}^{s,*}(\mathbb{Z}/p, \mathbb{Z}/p)\otimes \mathbb{Z}/p[v_3].
\]
To compute the $Ext$ groups
$Ext^*_{\mathbb{Z}/p[t_1, t_2,t_3]}(\mathbb{Z}/p, \mathbb{Z}/p)$, we can use
the modified May spectral sequence ($MSS$) introduced in \cite{ May64, May66, ravenelbook}.

    There is the May spectral sequence  $\{E_r^{s,t,\ast},\delta_r\}$ that converges to
$Ext^{s,t}_{\mathbb{Z}/p[t_1, t_2, t_3]}(\mathbb{Z}/p, \mathbb{Z}/p)$. The $E_1$-term of
this spectral sequence is
\begin{equation}\label{may}
E_1^{\ast,\ast,\ast}= E[h_{ij}|0\leqslant j, i=1,2,3]\otimes P[b_{ij}|0\leqslant j, i=1,2,3],
\end{equation}
where
\begin{align*}
h_{ij}\in & E_1^{1,q(1+p+\cdots+p^{i-1})p^j,2i-1} & &\mbox{and} & b_{ij}\in & E_1^{2,q(1+p+\cdots+p^{i-1})p^{j+1}, p(2i-1)}.
\end{align*}
The first May differential is given by
\begin{equation}\label{d}
\delta_1(h_{i,j})= \sum\limits_{0<k<i}^{}h_{i-k,k+j}h_{k,j}
\hskip .5cm \mbox{and} \hskip .5cm \delta_1(b_{i,j})=0.
\end{equation}

  For the reason of the total degree, to compute $Ext_{BP_*BP}^{s,s+(q(p^2+2p+2)-2)}(BP_*, BP_*/I_3)$
we only need to consider the sub-module generated by
$h_{30}, h_{20}, h_{10}, h_{21}, h_{11}, h_{12}$ and $b_{20}, b_{10}, b_{11}$, i.~e. the subcomplex
\[
E[h_{ij}|1\leqslant i, i+j\leqslant 3]\otimes E[b_{20}, b_{11}]\otimes P[b_{10}].
\]

   From (4.3), we know that within $t-s\leqslant q(p^2+2p+2)-2$ the May's $E_2$-term
\[
E_2^{s,*,*}=H^{s,*, *}(E_1^{s,*,*}, \delta_1)=H^{*,*,*}(E[h_{ij}|0\leqslant j, i+j\leqslant 3], \delta_1)
\otimes E[b_{20}, b_{11}]\otimes P[b_{10}].
\]
H.~Toda in \cite{Toda} computed the cohomology of $\left( E[h_{ij}|0\leqslant j, i+j\leqslant 3], \delta_1\right)$.
Here we only jot down the even dimensional elements within that range.

\[
\begin{array}{llllllllll}
h_{20}h_{10}, & q(p+2)-2; & & h_{20}h_{11}, & q(2p+1)-2; \\
h_{12}h_{10}, & q(p^2+1)-2; & & h_{21}h_{11}, & q(p^2+2p)-2.
\end{array}
\]
Thus within $t-s\leqslant q(p^2+2p+2)-2$, the even dimensional May's $E_2-$term $E_2^{s,t,*}$ is a submodule of
\[
\mathbb{Z}/p\{1, h_{20}h_{10}, h_{20}h_{11}, h_{12}h_{10}, h_{21}h_{11}\}\otimes E[b_{20}, b_{11}]\otimes P[b_{10}].
\]

  Suppose we have a generator $y$ in $Ext_{\mathbb{Z}/p[v_3][t_1, t_2, t_3]}^{s,s+q(p^2+2p+2)-2}(BP_*, BP_*/I_3)$.
Then $y$ is the form of $x$ or $v_3x$ where $x$ is an even dimensional generator in
$H^*(E[h_{ij}|i+j\leqslant 3])\otimes E[b_{20}, b_{11}]\otimes P[b_{10}]$.
\begin{enumerate}
\item If $y=v_3x$, then $x\in E_2^{s,t,*}$ subject to $t-s=q(p+1)-2$. An easy computation shows that the corresponding
$E_2$-term is zero.
\item If $y=x$, then $x\in E_2^{s,t,*}$ subject to $t-s=q(p^2+2p+2)-2$. Similarly, from
\begin{align*}
q(p^2+2p+2)-2\equiv &  6p-2 & & \hskip .5cm mod \hskip .3cm qp-2
\end{align*}
we compute that the total degree $t-s$ mod $qp-2$ of the generators in
\[
\mathbb{Z}/p\{1, h_{20}h_{10}, h_{20}h_{11}, h_{12}h_{10}, h_{21}h_{11}\}\otimes[b_{20}, b_{11}]
\]
and find non of them is $6p-2$. Thus the corresponding $E_2$-term is zero.
\end{enumerate}
The Lemma follows.
\end{proof}

It is easily showed that the following theorem holds from the lemma above.
\begin{theorem}\label{todabracket}For $p\geqslant 7$, $s\geqslant 1$, the
Toda bracket $\langle \alpha_1\beta_1, p,   \gamma_{s}\rangle =0$.
\end{theorem}

\begin{proof}
Let $\widetilde{v}_3$  be the composition of the following maps
\[
\xymatrix{S^{q(p^2+p+1)}\ar[r]^-{\widetilde{i}}&\Sigma^{q(p^2+p+1)}V(2)\ar[r]^-{v_3}& V(2),}
\]
where the first map is inclusion to the  bottom cell.

It is known that $\widetilde{v}_3$ is an order $p$ element in $\pi_{q(p^2+p+1)}V(2)$. Thus the Toda bracket
 $\langle \alpha_1\beta_1, p,  \widetilde{v}_3\rangle $ is well defined and
 $\langle \alpha_1\beta_1, p,  \widetilde{v}_3\rangle \in \pi_{q(p^2+2p+2)-2}V(2)=0$.
It follows that the Toda bracket  $\langle \alpha_1\beta_1, p,  \widetilde{v}_3\rangle=0.$

Let $\widetilde{j}:V(2)\longrightarrow S^{q(p+2)+3}$ be the collapsing lower cells map
from $V(2)$, then $\gamma_{s}  = \widetilde{v}_3\cdot v_3^{s-1} \cdot \widetilde{j}$.
As a result,
$$ \langle \alpha_1\beta_1, p,  \gamma_{s} \rangle =  \langle \alpha_1\beta_1, p,  \widetilde{v}_3\cdot v_3^{s-1}
\cdot \widetilde{j} \rangle =\langle \alpha_1\beta_1, p,  \widetilde{v}_3\rangle \cdot v_3^{s-1}
\cdot \widetilde{j}=0$$
because $\langle \alpha_1\beta_1, p,  \widetilde{v}_3\rangle =0 \in \pi_{q(p^2+2p+2)-2}V(2)=0.$
\end{proof}

 \begin{proposition} (see also \cite{ravenelbook1} 7.5.11)
Let $p\geqslant 7$ be an odd prime. Then in $\pi_*S^0$, the Toda bracket
$\langle \alpha_1\beta_1^{p-1}, \alpha_1\beta_1, p, \gamma_s\rangle$ is well defined and
\[
\alpha_1\beta_1^{p-1}h_{20}\gamma_s =\langle \alpha_1\beta_1^{p-1}, \alpha_1\beta_1, p, \gamma_s\rangle
= \beta_{p/p-1} \gamma_s.
\]
\end{proposition}

\begin{proof}
    From $\langle \beta_1^{p-1}, \alpha_1\beta_1, p \rangle=0$, $\langle \alpha_1\beta_1, p, \alpha_1\rangle=0$,
$\langle \alpha_1, \alpha_1\beta_1, p\rangle=0$ and $\langle \alpha_1\beta_1, p, \gamma_s\rangle=0$,
we know that the following 4-fold Toda bracket is well defined and
\[
\beta_{p/p-1} = \langle \beta_1^{p-1}, \alpha_1\beta_1, p,  \alpha_1\rangle;
\quad  \quad \alpha_1h_{20}\gamma_s=\langle \alpha_1, \alpha_1\beta_1, p,  \gamma_s\rangle.
 \]
On the other hand, one has
 \begin{eqnarray*}
  \beta_{1}^{p-1}\alpha_1h_{20}\gamma_s &=& \beta_{1}^{p-1}\langle \alpha_1, \alpha_1\beta_1, p,  \gamma_s\rangle \\
     &=& \langle \alpha_1 \beta_{1}^{p-1}, \alpha_1 \beta_1, p,  \gamma_s\rangle \\
     &=& \alpha_1 \langle \beta_{1}^{p-1}, \alpha_1 \beta_1, p,  \gamma_s\rangle \\
     &=& \langle  \beta_{1}^{p-1}, \alpha_1\beta_1, p, \alpha_1  \gamma_s\rangle \\
     &=& \langle  \beta_{1}^{p-1}, \alpha_1\beta_1, p,  \alpha_1\rangle \cdot \gamma_s \\
     &=& \beta_{p/p-1} \gamma_s
 \end{eqnarray*}
The proposition follows.
\end{proof}

\begin{theorem}\label{adifferential} Let $p\geqslant 7$ be an odd prime and $2 \leqslant s \leqslant p-2$.
Then in the $ANSS$, we have the following Adams-Novikov differential
\[
d_{2p-1}(h_{2,0}b_{1,1}\gamma_{s})=\alpha_1\beta_1^{p}h_{2,0}\gamma_{s}.
\]
\end{theorem}

\begin{proof}
Note that $b_{11}=\beta_{p/p}$. Then from (3.1) one has the differential in the small descent spectral sequence
($SDSS$)
\[
d_2(h_{20}b_{11})=\beta_1\beta_{p/p-1},
\]
which could be read as $d(h_{20}\beta_{p/p})=\beta_1\beta_{p/p-1}$ and
$d(h_{20}\beta_{p/p}\gamma_s)=\beta_1\beta_{p/p-1}\gamma_s$ in the cobar complex of $BP_*$ or equivalently
the first Adams-Novikov differential in the $ANSS$.
Then from the relation $\beta_{p/p-1}\gamma_s=\alpha_1\beta_1^{p-1}h_{20}\gamma_s$ in $\pi_*S^0$
and $\beta_{p/p-1}\gamma_s=0$ in $Ext^{5,*}_{BP_*BP}(BP_*, BP_*)$, we get the Adams differential
in the $ANSS$
\[
d_{2p-1}(h_{2,0}b_{1,1}\gamma_s)=\beta_1\cdot\beta_1^{p-1}\alpha_1h_{20}\gamma_s
=\alpha_1\beta_1^ph_{20}\gamma_s.
\]
The theorem follows.
\end{proof}

\section{The proof of Theorem A}

   In this section we  prove our main theorem by showing that
$\beta_{p^2/p^2-1}$ survives to $E_\infty$ in the $ANSS$. Note that $\beta_{p^2/p^2-1}$ has too
low dimension to be the target of an Adams-Novikov differential, we will do this by showing that all the Adams-Novikov differentials
$d_r(\beta_{p^2/p^2-1})$ are trivial.

\begin{lemma}
Let $p\geqslant 5$ and $i\not\equiv 0$ $mod$ $p$. In the $ANSS$, one has the following  Adams-Novikov differential
\[
d_{2p-1}(\eta_{i})=\beta_1^{p}\beta_{i+1}
\]
\end{lemma}

\begin{proof}
   Recall from \cite{ravenelbook1} 7.3.11 Theorem (e), in the $SDSS$
\[
E_1=Ext_{BP_*BP}^{s,t}(BP_*, BP_*X^{p^2-1})\otimes E[h_{11}]\otimes P[b_{11}] \Longrightarrow Ext_{BP_*BP}^{s,t}(BP_*, BP_*X),
\]
where $BP_*X^{p^2-1}=BP_*[t_1]/\langle t_1^{p^2} \rangle$ (\cf \cite{ravenelbook1} 7.3.8 Theorem),
one has $d_2(h_{20}\mu_{i-1})=ib_{11}\beta_{i+1}$. And from its definition
we know that $\eta_i=h_{11}\mu_{i-1}$ is represented by
\[
\delta\delta \left(\frac{v_2^{p+i-1}t_2+v_2^it_2^p-v_2^{i}t_1^{p^2+p}-v_2^{i-1}v_3t_1^p}{pv_1}\right)
\]
(\cf \cite{ravenelbook1} p.288) which is also denoted by $\displaystyle{\delta\delta\left(\frac{v_2^{p+i}}{pv_1}\zeta_2\right)}$ in
 \cite{miller,xwang}.
In the cobar complex of $N_0^2=BP_*/(p^\infty, v_1^\infty)$, a straightforward computation shows that the coboundary of
\begin{align*}
 & \frac{v_2^i(t_3-t_1t_2^p-t_2t_1^{p^2}+t_1^{p^2+p+1})+v_2^{p+i-1}(t_1t_2-t_1^{p+2})-v_2^{i-1}v_3(t_2-t_1^{p+1})}{pv_1}\\
+&\frac{2v_2^{p+i}}{(p+i)p^2v_1}t_1-\frac{v_2^{p+i}}{(p+i)pv_1^2}t_1^2
\end{align*}
is $\displaystyle{\frac{(v_2^{p+i-1}t_2+v_2^it_2^p-v_2^it_1^{p^2+p}-v_2^{i-1}v_3t_1^p)\otimes t_1}{pv_1}+\frac{v_2^{i+1}}{pv_1}b_{11}}$.
This shows that in $Ext_{BP_*BP}^{2,*}(BP_*, N_0^2)$ the cohomology class
\[
\left[\frac{(v_2^{p+i-1}t_2+v_2^it_2^p-v_2^it_1^{p^2+p}-v_2^{i-1}v_3t_1^p)\otimes t_1}{pv_1}\right]=
-\left[\frac{v_2^{i+1}}{pv_1}b_{11}\right].
\]
Applying the connecting homomorphism $\delta\delta$, we get $\alpha_1\eta_i=\beta_{i+1}\beta_{p/p}$.

  From $\alpha_1\eta_{i}=\beta_{i+1}\beta_{p/p}$ and the Toda differential, one has:
\[
\alpha_1 d_{2p-1}(\eta_{i}) = d_{2p-1}(\alpha_1\eta_{i})=d_{2p-1}(\beta_{i+1}\beta_{p/p})=\alpha_1\beta_1^p\beta_{i+1}
\]
The lemma follows from $\alpha_1d_{2p-1}(\eta_{i})=\alpha_1\beta_1^p\beta_{i+1}$.
\end{proof}

{\bf Proof of Theorem A} From $\beta_{p^2/p^2-1}\in Ext^{2,q(p^3+1)}_{BP_*BP}(BP_*, BP_*)$, we know that
$d_r(\beta_{p^2/p^2-1})\in Ext^{s,t}_{BP_*BP}(BP_*, BP_*)$ subject to $t-s=q(p^3+1)-3$.  From Theorem 3.2 we know that
the corresponding $Ext^{s,t}_{BP_*BP}(BP_*, BP_*)$ is the $\mathbb{Z}/p$-module generated by
$\mathfrak{g}_1, \mathfrak{g}_3, \mathfrak{g}_4, \mathfrak{g}_6$ and $\mathfrak{g}_7, \mathfrak{g}_8$.

  $\mathfrak{g}_7=\alpha_1\beta_{(p-1)p+1}$ and $\mathfrak{g}_8=\alpha_1\beta_{p^2/p^2}$ have too low
dimension to be the target of $d_r(\beta_{p^2/p^2-1})$.

  From the Toda differential $d_{2p-1}(b_{11})=\alpha_1\beta_1^p$ we have
\begin{align*}
d_{2p-1}(\beta_1^{p^2-p-1}b_{11}\beta_2) = & \alpha_1\beta_1^{p^2-1}\beta_2 =\mathfrak{g}_1\\
d_{2p-1}(\mathfrak{g}_4)=d_{2p-1}(\beta_1^{\frac{p^2-6p+1}{2}}b_{11}^2\gamma_{\frac{p+1}{2}}) = &
2\alpha_1\beta_1^{\frac{p^2-4p+1}{2}}b_{11}\gamma_{\frac{p+1}{2}}.
\end{align*}

  From $d_{2p-1}(h_{20}b_{11}\gamma_s)=\alpha_1\beta_1^ph_{20}\gamma_s$ (\cf Theorem 4.4), we have
\[
d_{2p-1}\left(\beta_1^{\frac{p^2-4p-1}{2}}h_{20}b_{11}\gamma_{\frac{p+1}{2}}\right)
  = \alpha_1\beta_1^{\frac{p^2-2p-1}{2}}h_{20}\gamma_{\frac{p+1}{2}}=\mathfrak{g}_3.
\]

  From Lemma 5.1, we have
\[
d_{2p-1}(\mathfrak{g}_6)=d_{2p-1}(\beta_1^{p-1}\eta_{(p-3)p+3})=\beta_1^{2p-1}\beta_{(p-3)p+4}.
\]

{\small
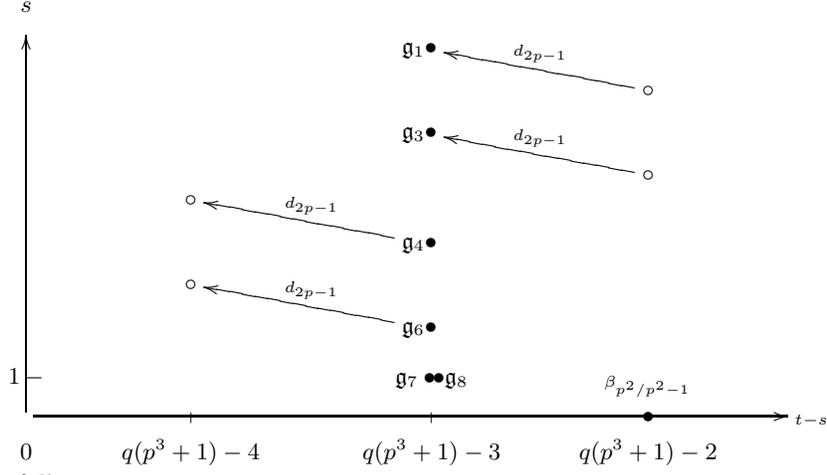
\xymatrix@R=0pt{
&&s&&&& \\
&&&&&&\\
&&&&\mathfrak{g}_1\bullet&&\\
&&&&&&\\
&&&&&\circ\ar[uul]_{d_{2p-1}}&\\
&&&&&&\\
&&&&\mathfrak{g}_3\bullet&&\\
&&&&&&\\
&&&&&\circ\ar[uul]_{d_{2p-1}}&\\
&&&\circ&&&\\
&&&&&&\\
&&&&\mathfrak{g}_4\bullet\ar[uul]_{d_{2p-1}}&&\\
&&&&&&\\
&&&\circ&&&\\
&&&&&&\\
&&&&\mathfrak{g}_6\bullet\ar[uul]_{d_{2p-1}}&&\\
&&&&&&\\
&&1\ \! -&&\quad\mathfrak{g}_7\bullet\!\!\bullet \mathfrak{g}_8 & _{\beta_{p^2/p^2-1}}&\\
&&\ar[uuuuuuuuuuuuuuuuu]\ar[rrrr]& \shortmid&\quad \shortmid& \bullet & _{t-s}\\
&&0&q(p^3+1)-4&\quad q(p^3+1)-3&q(p^3+1)-2&}
}
Theorem A follows. \qb

 Consider the cofiber sequence
\[
\xymatrix{
S^0\ar[r]^p & S^0 \ar[r] & M
}
\]
and the induced  long exact sequence of $Ext$ groups
\[
\xymatrix @C=0.6cm @R=0.8cm{
\cdots \ar[r] & Ext_{BP_*BP}^{1,t}(BP_*, BP_*M) \ar[r]^\delta\ar[d]^{d_{2p-1}} &
  Ext_{BP_*BP}^{2,t}(BP_*, BP_*S^0) \ar[r]\ar[d]^{d_{2p-1}} & \cdots\\
\cdots \ar[r] & Ext_{BP_*BP}^{2p,*}(BP_*, BP_*M) \ar[r]^\delta &
  Ext_{BP_*BP}^{2p+1,*}(BP_*, BP_*S^0) \ar[r] & \cdots .
}
\]
One has $\delta(h_{i+2})=\beta_{p^{i+1}/p^{i+1}}$, $\delta(v_1h_{i+2})=\beta_{p^{i+1}/p^{i+1}-1}$ and $\delta(v_1^i)=i\alpha_i$.
From the Toda differential
$d_{2p-1}(\beta_{p/p})=\alpha_1\beta_1^p$, one can get a non-trivial Adams-Novikov differential in the $ANSS$ for the
Moore spectrum $M$
\[
d_{2p-1}(h_2)=v_1\beta_1^p.
\]
From the relation $h_{i+1}\beta_{p/p}^{p^i}=h_{i+2}\beta_1^{p^i}$ (\cf \cite{ravenelbook1} 6.4.7), we get the
following Adams-Novikov differential by induction
\begin{align*}
d_{2p-1}(h_{i+2})= & v_1\beta_{p^i/p^i}^p  & &\mbox{and then} & d_{2p-1}(v_1h_{i+2}) = & v_1^2\beta_{p^i/p^i}^p.
\end{align*}
Then applying the connecting homomorphism $\delta$, we have
\begin{align}\tag{5.2}
d_{2p-1}(\beta_{p^{i+1}/p^{i+1}-1})= d_{2p-1}(\delta(v_1h_{i+2}))=\delta(d_{2p-1}(v_1h_{i+2}))=\delta(v_1^2\beta_{p^i/p^i}^p)
=2\alpha_2\beta_{p^i/p^i}^p.
\end{align}

    Unfortunately we could not prove that $\alpha_2\beta_{p^i/p^i}^p$ is non-zero in $Ext_{BP_*BP}^{2p+1,*}(BP_*, BP_*)$.
Indeed $\alpha_2\beta_1^p=0$ from $\alpha_2\beta_1=0$. From Theorem A we know that $\beta_{p^2/p^2-1}$ survives to
$E_\infty$ in the $ANSS$. Thus from $d_{2p-1}(\beta_{p^2/p^2-1})=\alpha_2\beta_{p/p}^p$, we know that
$\alpha_2\beta_{p/p}^p$ MUST be zero in $Ext_{BP_*BP}^{2p+1,*}(BP_*, BP_*)$.

   As we know that $\beta_{p/p}^p\not=0$ in $Ext_{BP_*BP}^{2p,qp^3}(BP_*, BP_*)$, but $b_{11}^p=0$
in $Ext_{BP_*BP}^{2p,qp^3}(BP_*, BP_*X)$ (\cf \cite{ravenelbook1} 7.3.12 (b)). So there must be an element that converges to $\beta_{p/p}^p$ in the
$SDSS$. In the case $p=5$, $\beta_1x_{952}$ converges $\beta_{5/5}^5$ and this implies
$\alpha_2\beta_{5/5}^5=\alpha_2\beta_1x_{952}=0$.

{\bf Conjecture:} For $i\geqslant 1$, $\beta_{p^i/p^i}^p\not=0$ in $Ext_{BP_*BP}^{2p,qp^{i+2}}(BP_*, BP_*)$. But $b_{1i}^p=0$ in \\
$Ext_{BP_*BP}^{2p, qp^{i+2}}(BP_*, BP_*X^{p^i-1})$, where $X^{p^i-1}\hookrightarrow T(1)$ is characterized by
\[
BP_*X^{p^i-1}=BP_*\{1,t_1, t_1^2, \cdots, t_1^{p^i-1}\}.
\]
 Thus in the $SDSS$
\begin{align*}
Ext_{BP_*BP}^{s,t}(BP_*, BP_*X^{p^i-1})\otimes E[h_{1i-1}]\otimes P[b_{1i-1}] & \Longrightarrow Ext_{BP_*BP}^{*,*}(BP_*, BPX^{p^{i-1}-1})\\
Ext_{BP_*BP}^{s,t}(BP_*, BP_*X^{p^{i-1}-1})\otimes E[h_{1i-2}]\otimes P[b_{1i-2}] & \Longrightarrow Ext_{BP_*BP}^{*,*}(BP_*, BPX^{p^{i-2}-1})\\
                                                                              & \vdots \\
Ext_{BP_*BP}^{s,t}(BP_*, BP_*X)\otimes E[\alpha_1]\otimes P[\beta_1] & \Longrightarrow Ext_{BP_BP}^{2p,qp^{i+2}}(BP_*, BP_*)
\end{align*}
there is a $\zeta_i\in Ext_{BP_*BP}^{2p-2, q(p^{i+2}-p^i)}(BP_*, BP_*X^{p^i-1})$ such that $b_{1i-1}\zeta_i$ converges to
$\beta_{p^i/p^i}^p$. This implies that:
\begin{enumerate}
\item For $i=1$, $\alpha_2\beta_{p/p}^p=\alpha_2\beta_1\zeta_1=0$ in $Ext_{BP_*BP}^{2p, qp^3}(BP_*, BP_*)$.
\item For $i>1$, $\alpha_2\beta_{p^i/p^i}^p=\alpha_2\beta_{p^{i-1}/p^{i-1}}\zeta_i=\alpha_1\beta_{p^{i-1}/p^{i-1}-1}\zeta_i$,
which is NOT zero in\\ $Ext_{BP_*BP}^{2p, qp^{i+2}}(BP_*, BP_*)$. 
\end{enumerate}
Then from (5.2) we know that for $i>1$, $\beta_{p^{i+1}/p^{i+1}-1}$ does not exist.

\end{document}